\pdfoutput=1
\documentclass[11pt]{article}

\usepackage[T1]{fontenc}
\usepackage[utf8]{inputenc}
\input{glyphtounicode}
\pdfglyphtounicode{parenleftbig}{0028}
\pdfglyphtounicode{parenrightbig}{0029}
\pdfglyphtounicode{parenleftBig}{0028}
\pdfglyphtounicode{parenrightBig}{0029}
\pdfglyphtounicode{parenleftbigg}{0028}
\pdfglyphtounicode{parenrightbigg}{0029}
\pdfglyphtounicode{parenleftBigg}{0028}
\pdfglyphtounicode{parenrightBigg}{0029}
\pdfglyphtounicode{bracketleftbig}{005B}
\pdfglyphtounicode{bracketrightbig}{005D}
\pdfglyphtounicode{braceleftbig}{007B}
\pdfglyphtounicode{bracerightbig}{007D}
\pdfglyphtounicode{slashbig}{002F}
\pdfglyphtounicode{backslashbig}{005C}
\usepackage{amsmath,amssymb,amsthm,mathtools}
\usepackage{microtype}
\usepackage[colorlinks=true,linkcolor=blue,citecolor=blue,urlcolor=blue]{hyperref}
\hypersetup{pdfauthor={Yawara Ishida},pdftitle={No involutions in the missing Moore graph}}
\title{No involutions in the missing Moore graph}

\author{Yawara Ishida\\[3pt]
  {\small\normalfont\itshape A.I.\ System Research, Inc., Kyoto, Japan}\\[2pt]
  {\small\normalfont\texorpdfstring{\href{mailto:yawara@aisr.dev}{yawara@aisr.dev}}{yawara@aisr.dev}}}
\date{}

\theoremstyle{plain}
\newtheorem{theorem}{Theorem}[section]
\newtheorem{corollary}[theorem]{Corollary}
\newtheorem{lemma}[theorem]{Lemma}
\newtheorem{proposition}[theorem]{Proposition}

\theoremstyle{remark}
\newtheorem{remark}[theorem]{Remark}

\newcommand{\Aut}{\operatorname{Aut}}
\newcommand{\Fix}{\operatorname{Fix}}
\newcommand{\tr}{\operatorname{tr}}
\newcommand{\rank}{\operatorname{rank}}
\newcommand{\Tr}{\operatorname{Tr}}
\newcommand{\Zp}{\mathbb Z_p}
\newcommand{\Qp}{\mathbb Q_p}
\newcommand{\Fp}{\mathbb F_p}
\newcommand{\Ftwo}{\mathbb F_2}
\newcommand{\one}{\mathbf 1}

\begin{document}

\maketitle

\begin{abstract}
The existence of a Moore graph of degree $57$ is the last open case in the Hoffman--Singleton classification of Moore graphs of diameter two.
We prove that such a graph, if one exists, has no involutory automorphisms; consequently its automorphism group has odd order.
The proof relies on a new trace--rank identity,
valid for every finite graph and obtained from the Brauer quotient of $p$-permutation lattices:
the trace of an automorphism of prime order $p$ on a spectral summand whose projection idempotent preserves $p$-adic permutation lattice equals
the rank, over the residue field of characteristic $p$, of the corresponding idempotent restricted to the fixed vertices.
For strongly regular graphs,
this upgrades the classical character-value congruences for prime-order automorphisms to exact rank equalities.
As further applications, the identity yields new restrictions on automorphisms of odd prime order of the hypothetical Moore graph,
and it sharpens known automorphism analyses of other hypothetical strongly regular graphs.
\end{abstract}

\section{Introduction}

Moore graphs are the extremal graphs for the degree--diameter problem: regular graphs whose order attains the Moore bound for their degree and diameter.
We refer to the surveys of Miller--\v{S}ir\'a\v{n} \cite{MillerSiran} and Dalf\'o \cite{Dalfo} for general background.
In diameter two, the theorem of Hoffman and Singleton shows that the degree of such a graph can only be
\[
 2,\quad 3,\quad 7,\quad 57;
\]
the first three cases are the pentagon, the Petersen graph, and the Hoffman--Singleton graph,
while the case of degree $57$ remains open \cite{HoffmanSingleton}.
For diameter greater than two,
Bannai--Ito \cite{BannaiIto} and Damerell \cite{Damerell} independently proved that the only Moore graphs are the odd cycles,
so the degree-$57$ graph is the unique Moore graph whose existence is undecided.
The hypothetical graph in the last case is often called the missing Moore graph.
It would be a strongly regular graph with parameters \cite{HoffmanSingleton,BrouwerHaemers}
\[
 (v,k,\lambda,\mu)=(3250,57,0,1).
\]
Recent work on the existence question includes the study of integral invariants such as the critical group \cite{Ducey},
structural constraints on the hypothetical graph \cite{Renteln,RentelnCorrigendum},
and computational search \cite{SmithMontemanni2023,SmithMontemanni2024,SmithMontemanni2026}.

Even in the absence of an existence proof, the possible automorphisms of the missing Moore graph are subject to strong restrictions.
Higman's unpublished argument, as recorded in Brouwer--Haemers and discussed in Cameron's text on permutation groups,
shows that such a graph cannot be vertex-transitive \cite{BrouwerHaemers,Cameron}.
Makhnev--Paduchikh studied the fixed-point subgraphs of automorphisms and the case of even group order \cite{MakhnevPaduchikh,MakhnevPaduchikh2009};
Ma\v{c}aj--\v{S}ir\'a\v{n} proved that the order of the automorphism group is at most $375$ if that order is odd and at most $110$ if it is even \cite{MacajSiran};
and Kov\'a\v{c}ikov\'a, in a doctoral thesis, refined the resulting restrictions for automorphisms of order $7$ \cite{Kovacikova}.
The main result of this paper eliminates the even-order case altogether.

\begin{theorem}\label{thm:main}
Let $\Gamma$ be a Moore graph of diameter two and degree $57$.
Then $\Aut(\Gamma)$ contains no element of order $2$.
\end{theorem}

Consequently, if the missing Moore graph exists, its automorphism group has odd order (Corollary \ref{cor:odd-order}).

The proof is based on a trace--rank identity for idempotent summands of permutation lattices.
Let $x$ be an automorphism of prime order $p$ of a finite graph $\Delta$, let $\mathcal F$ be its set of fixed vertices,
and let $O$ be a complete discrete valuation ring of residue characteristic $p$,
with field of fractions $K$, maximal ideal $\mathfrak m$, and residue field $\kappa$.
If $E$ is an idempotent endomorphism of the vertex lattice $O^{V(\Delta)}$ commuting with the action of $x$,
with image $L=EO^{V(\Delta)}$, then
\[
 \tr\bigl(x\mid K\otimes_OL\bigr)
 =\rank_{\kappa}\bigl(E[\mathcal F]\bmod\mathfrak m\bigr),
\]
where $E[\mathcal F]$ is the principal submatrix of $E$ on the fixed vertices (Corollary \ref{cor:graph}).
The identity is obtained by passing to the Brauer quotient of the $p$-permutation lattice $L$, reviewed in Section \ref{sec:padic};
in particular it applies to every spectral idempotent of the adjacency matrix that preserves the $p$-adic vertex lattice.

For strongly regular graphs the identity becomes especially concrete.
The spectral idempotents are linear combinations
\[
 E_\theta=\alpha_\theta I+\beta_\theta A+\gamma_\theta J,
\]
and whenever $E_\theta$ preserves the lattice, the identity reads
\[
 \alpha_\theta\, a_0(x)+\beta_\theta\, a_1(x)+\gamma_\theta\, v
 =\rank_{\kappa}\bigl((\alpha_\theta I_{\mathcal F}+\beta_\theta A_{\mathcal F}+\gamma_\theta J_{\mathcal F})\bmod \mathfrak m\bigr),
\]
where $v$ is the number of vertices and
\[
 a_i(x)=\left|\{u\in V(\Delta):d(u,x(u))=i\}\right|
\]
are the displacement counts of $x$, $d$ denoting the graph distance (Proposition \ref{prop:srg-rank}).
Thus the ordinary character value on the left,
whose integrality is the source of the classical congruences for $a_1(x)$ \cite{MacajSiran,DeWinterKamischkeWang},
is not merely constrained: it is equal to an explicit modular rank determined by the fixed-point subgraph (Remark \ref{rem:benson}).
Rationality of the eigenvalues is not required: if $E_\theta$ is defined only over a finite extension of $\Qp$,
the same argument is carried out over the valuation ring of that extension (Remark \ref{rem:integral}).

For the missing Moore graph the non-principal eigenvalues are $7$ and $-8$.
Section \ref{sec:spectral} computes the two corresponding spectral idempotents and determines the primes at which each preserves the $p$-adic vertex lattice;
at $p=2$, only the summand for the eigenvalue $-8$ is available.
Combining the resulting rank equality with the fixed-point theorem for involutions
--- the fixed subgraph of an involution $t$ is the star $K_{1,55}$ (Theorem \ref{thm:higman}) ---
forces a negative value of the count $a_1(t)$, a contradiction.
This proves Theorem \ref{thm:main} (Section \ref{sec:involutions}).

The same comparison yields new restrictions for automorphisms of odd prime order (Section \ref{sec:odd}).
For the primes at which the known fixed-point alternatives are precise enough, we obtain the following.

\begin{theorem}\label{thm:odd}
Let $\Gamma$ be a Moore graph of diameter two and degree $57$,
and let $x\in\Aut(\Gamma)$ be an automorphism of prime order $p\in\{7,11,13,19\}$.
Then $a_0(x)$ and $a_1(x)$ satisfy the following restrictions:
\[
\begin{array}{c|c|c}
 p & a_0(x) & a_1(x) \\
\hline
 7 & 2,\ 9,\ 16,\ 23,\ 30,\ 37 & 7a_0(x)+35 \\
 11 & 5 & 55 \\
 13 & 0 & 65 \\
 19 & 1 & 57
\end{array}
\]
\end{theorem}

For $p=7$,
Theorem \ref{thm:odd} recovers a theorem of Kov\'a\v{c}ikov\'a \cite[Theorem 5.1]{Kovacikova} whose proof rests on computer-assisted counts of induced subgraphs;
the derivation given here uses neither the computer search nor the tabulated bounds of \cite[Lemma 12]{MacajSiran}
(see Section \ref{subsec:order7} for a detailed comparison).
For $p=11$, $13$, and $19$,
Theorem \ref{thm:odd} selects the smallest of the values of $a_1(x)$ allowed by the table of Ma\v{c}aj--\v{S}ir\'a\v{n} \cite[Lemma 12]{MacajSiran},
which appears to be new.
The remaining odd primes that can occur as orders of automorphisms of $\Gamma$ are $3$ and $5$ (Lemma \ref{lem:fix-odd}): for these,
neither spectral idempotent of $\Gamma$ is defined over the $p$-adic integers,
and the method of this paper gives nothing.

The trace--rank identity is independent of the Moore graph parameters.
Beyond the missing Moore graph,
Section \ref{sec:applications} applies it to sharpen known automorphism analyses of five hypothetical strongly regular graphs,
replacing character congruences by exact displacement counts.

\section{Known facts on the missing Moore graph}\label{sec:known}

Throughout the paper, $\Gamma$ denotes a Moore graph of diameter two and degree $57$, and $V(\Gamma)$ denotes its vertex set.
The graph is finite, simple, undirected, and regular.

\subsection{Parameters and spectrum}

The Moore bound gives
\[
 |V(\Gamma)|=1+57^2=3250,
\]
and the diameter-two Moore graph condition implies that $\Gamma$ is strongly regular.
A strongly regular graph with parameters $(v,k,\lambda,\mu)$ is a $k$-regular graph on $v$ vertices in which any two
adjacent vertices have exactly $\lambda$ common neighbours and any two distinct non-adjacent vertices have exactly $\mu$ common neighbours.
For $\Gamma$ these parameters are \cite{HoffmanSingleton,BrouwerHaemers}
\[
(v,k,\lambda,\mu)=(3250,57,0,1);
\]
here $\lambda=0$ because $\Gamma$ is triangle-free,
and $\mu=1$ because any two non-adjacent vertices have a unique common neighbour.
Let $A$ be its adjacency matrix and let $J$ be the all-one matrix.  Then
\begin{equation}\label{eq:srg-relation}
 A^2+A=56I+J.
\end{equation}
Consequently the spectrum of $A$ is
\begin{equation}\label{eq:spectrum}
 \operatorname{Spec}(A)=\{57^1,\ 7^{1729},\ (-8)^{1520}\}.
\end{equation}
This spectral calculation is standard for Moore graphs and strongly regular graphs with these parameters \cite{HoffmanSingleton,BrouwerHaemers}.

\subsection{Fixed-point structure}\label{subsec:fix}

The fixed-point calculation for involutions is part of Higman's unpublished argument and was further used by Makhnev--Paduchikh in their study of $\Gamma$.
For an automorphism $g$, write $\Fix(g)$ for the subgraph induced by the vertices fixed by $g$.
We need the following precise form.

\begin{theorem}[Higman--Makhnev--Paduchikh]\label{thm:higman}
Let $t\in\Aut(\Gamma)$ be an involution.  Then
\[
 \Fix(t)\cong K_{1,55}.
\]
In particular, $t$ fixes exactly $56$ vertices.
\end{theorem}

The assertion that an involution fixes exactly $56$ vertices, and that its fixed vertices form a star,
is proved in Cameron's account of Higman's argument \cite[proof of Theorem 3.13]{Cameron};
the count is also contained in the proof of Brouwer--Haemers \cite[Proposition 11.5.2]{BrouwerHaemers}.
The star statement is recorded explicitly by Makhnev--Paduchikh \cite[Lemma 4]{MakhnevPaduchikh},
who attribute the result to Higman and cite Cameron's account for its proof;
it recurs in the fixed-point alternative of their later paper \cite[Proposition 1]{MakhnevPaduchikh2009}.
We use only this fixed-point statement and not any classification of possible automorphism groups.

For automorphisms of odd prime order we use the following alternative,
due to Makhnev--Paduchikh \cite[Lemma 3]{MakhnevPaduchikh} and stated in the form below as Lemma 4 of Ma\v{c}aj--\v{S}ir\'a\v{n} \cite{MacajSiran},
here phrased for a single automorphism.
In \cite{MacajSiran} each case carries a divisibility condition on the order of the group,
which for a group of prime order $p$ amounts to the memberships stated below;
the underlying six-fold alternative for fixed-point subgraphs goes back to Aschbacher's analysis \cite{Aschbacher}.

\begin{lemma}[Makhnev--Paduchikh; Ma\v{c}aj--\v{S}ir\'a\v{n}]\label{lem:fix-odd}
Let $x\in\Aut(\Gamma)$ have odd prime order $p$.  Then one of the following holds:
\begin{enumerate}
\item[(1)] $\Fix(x)$ is empty and $p\in\{5,13\}$;
\item[(2)] $\Fix(x)$ is a single vertex and $p\in\{3,19\}$;
\item[(3)] $\Fix(x)$ is a star $K_{1,1+7\ell}$ with $\ell\ge 0$ and $p=7$;
\item[(4)] $\Fix(x)$ is a pentagon and $p\in\{5,11\}$;
\item[(5)] $\Fix(x)$ is the Petersen graph and $p=3$;
\item[(6)] $\Fix(x)$ is the Hoffman--Singleton graph and $p=5$.
\end{enumerate}
\end{lemma}

We also record the elementary properties of the displacement counts $a_i(x)$ defined in the introduction: they are non-negative integers,
$a_0(x)$ is the number of vertices of $\Fix(x)$, and
\[
 a_0(x)+a_1(x)+a_2(x)=3250.
\]
Our notation agrees with that of \cite{MacajSiran} and \cite{Kovacikova}.

\section{\texorpdfstring{A trace formula for $p$-permutation lattices}{A trace formula for p-permutation lattices}}\label{sec:padic}

We record the module-theoretic terminology used below.
Let $O$ be a complete discrete valuation ring with maximal ideal $\mathfrak m$,
residue field $\kappa=O/\mathfrak m$ of characteristic $p>0$, and field of fractions $K$ of characteristic $0$,
and let $G$ be a finite group.
An $OG$-module is a module over the group algebra $OG$.
An $OG$-lattice is an $OG$-module which is finitely generated and free as an $O$-module.
A permutation $OG$-lattice is an $OG$-lattice admitting an $O$-basis permuted by $G$.

An $OG$-lattice is called $p$-permutation if its restriction to a Sylow $p$-subgroup of $G$ is a permutation lattice.
Equivalently, it admits an $O$-basis permuted by a Sylow $p$-subgroup.
We shall use the standard characterization that direct summands of permutation $OG$-lattices are $p$-permutation;
see Lassueur \cite[Proposition--Definition 4.2]{Lassueur}.
The Brauer quotient and its basis description for permutation modules are used in the form given by Brou\'e and by Bouc--Th\'evenaz \cite{Broue,BoucThevenaz}.

We use this material for
\[
 G=C_p=\langle x\rangle,\qquad x^p=1,
\]
cyclic of prime order equal to the residue characteristic $p$ of $O$.
In this case, a $p$-permutation $OC_p$-lattice is simply an $OC_p$-lattice admitting an $O$-basis permuted by $C_p$,
because $C_p$ is its own Sylow $p$-subgroup.
For an $OC_p$-lattice $N$, put
\[
 \overline N=\kappa\otimes_{O}N\cong N/\mathfrak mN.
\]
We write
\begin{equation}\label{eq:brauer-quotient}
 N(C_p)=\overline N^{C_p}\big/\Tr_1^{C_p}(\overline N),
\end{equation}
where
\[
 \Tr_1^{C_p}(v)=v+x(v)+\dots+x^{p-1}(v)
\]
is the relative trace from the trivial subgroup to $C_p$.
Thus $N(C_p)$ is the usual Brauer quotient of the $\kappa C_p$-module $\overline N$ at the subgroup $C_p$,
written so as to record the lattice $N$ from which it comes.

The following observation shows that, for a lattice with a basis permuted by $C_p$,
the ordinary trace in characteristic zero and the Brauer quotient in characteristic $p$ count the same fixed basis elements.

\begin{lemma}\label{lem:trace-brauer}
Let $N$ be a $p$-permutation $OC_p$-lattice.  Then
\[
 \tr\left(x\mid K\otimes_{O}N\right)=\dim_{\kappa}N(C_p),
\]
where the integer on the right is regarded as an element of $K$.
\end{lemma}

\begin{proof}
Since $N$ is $p$-permutation and $C_p$ is its own Sylow $p$-subgroup, $N$ has an $O$-basis $\mathcal X$ permuted by $C_p$.
As $C_p$ has prime order, the $C_p$-orbits on $\mathcal X$ have size one or $p$.

A fixed basis element contributes $1$ to the trace of $x$ on $K\otimes_{O}N$.
An orbit of size $p$, on which $x$ acts as a $p$-cycle, contributes the trace of the permutation matrix of a $p$-cycle,
which is $0$.
Hence the left-hand side counts the fixed elements of $\mathcal X$.

After reduction modulo $\mathfrak m$,
the image $\overline{\mathcal X}$ of $\mathcal X$ is a $\kappa$-basis of $\overline N$ permuted by $C_p$.
If $\{b,x(b),\dots,x^{p-1}(b)\}$ is an orbit of size $p$ in $\mathcal X$,
then the invariant vectors supported on it are spanned by
\[
 \overline b+x(\overline b)+\dots+x^{p-1}(\overline b)=\Tr_1^{C_p}(\overline b),
\]
which vanishes in $N(C_p)$.
If $b\in\mathcal X$ is fixed by $C_p$, its image survives in the quotient.
The images of the fixed elements of $\mathcal X$ therefore form a basis of $N(C_p)$,
as in the standard basis description of the Brauer quotient for permutation modules \cite[Proposition 2.5(4)]{BoucThevenaz}.
Hence the right-hand side also counts the fixed elements of $\mathcal X$.
\end{proof}

\section{The trace--rank identity}\label{sec:general}

Throughout this section, $O$ is a complete discrete valuation ring as in Section \ref{sec:padic},
with maximal ideal $\mathfrak m$, residue field $\kappa$ of characteristic $p>0$,
and field of fractions $K$ of characteristic $0$.

\subsection{Idempotent summands of permutation lattices}

Let $X$ be a finite set, let $x$ be a permutation of $X$ of order $p$, and let
\[
 M=O^X
\]
be the corresponding permutation $OC_p$-lattice, where $C_p=\langle x\rangle$, and write $\overline M=\kappa\otimes_OM$.
Let $\mathcal F\subseteq X$ be the set of fixed points of $x$.

The Brauer quotient of $M$ has a basis indexed by the fixed points.
Indeed, if $w\in X$ is fixed by $x$,
then the image $\overline e_w$ of the standard basis vector is $C_p$-invariant and survives in $M(C_p)$.
If $w$ is not fixed, its $C_p$-orbit has size $p$, and the invariant vectors supported on this orbit are spanned by the orbit sum
\[
 \overline e_w+\overline e_{x(w)}+\dots+\overline e_{x^{p-1}(w)}=\Tr_1^{C_p}(\overline e_w),
\]
which is zero in the Brauer quotient.  Hence
\begin{equation}\label{eq:MC-basis}
 M(C_p)\cong \kappa^{\mathcal F}.
\end{equation}

\begin{theorem}\label{thm:local}
In the situation above, let $E\in\operatorname{End}_O(M)$ be an idempotent commuting with the action of $C_p$,
set $L=EM$,
and let $E[\mathcal F]$ denote the principal submatrix of the matrix of $E$ in the standard basis of $M$ on the rows and columns indexed by $\mathcal F$.
Then
\[
 \tr\bigl(x\mid K\otimes_O L\bigr)=\rank_{\kappa}\bigl(E[\mathcal F]\bmod \mathfrak m\bigr),
\]
the integer on the right being regarded as an element of $K$.
\end{theorem}

\begin{proof}
Since $M$ is free over the discrete valuation ring $O$, the image $L=EM$ of the idempotent $E$ is an $OC_p$-lattice,
and it is a direct summand of the permutation $OC_p$-lattice $M$.
By the characterization of $p$-permutation lattices recalled in Section \ref{sec:padic},
$L$ is therefore a $p$-permutation $OC_p$-lattice, and Lemma \ref{lem:trace-brauer} gives
\[
 \tr\bigl(x\mid K\otimes_O L\bigr)=\dim_\kappa L(C_p).
\]

It remains to compute this dimension.
The reduction $\overline E$ maps $C_p$-fixed vectors to $C_p$-fixed vectors and relative traces to relative traces,
because $E$ commutes with the action of $C_p$; therefore $E$ induces an idempotent endomorphism $E(C_p)$ of $M(C_p)$.
Since
\[
 M=L\oplus (1-E)M
\]
as $OC_p$-lattices, and since the Brauer quotient is additive on direct sums, this decomposition induces
\[
 M(C_p)=L(C_p)\oplus ((1-E)M)(C_p),
\]
and $E(C_p)$ is the projection onto the first summand.  Hence
\begin{equation}\label{eq:imageEC}
 \operatorname{im}E(C_p)=L(C_p).
\end{equation}

Under the identification \eqref{eq:MC-basis}, the matrix of $E(C_p)$ is $E[\mathcal F]\bmod\mathfrak m$.
To see this, apply $\overline E$ to $\overline e_w$ for a fixed point $w$.
Since $E$ commutes with the action of $x$ and $x(w)=w$,
the coefficients of $\overline E\,\overline e_w$ are constant on each $C_p$-orbit;
the component supported on an orbit of size $p$ is therefore a scalar multiple of the orbit sum,
i.e.\ a relative trace, and disappears in the Brauer quotient.
Thus only the coordinates at fixed points remain.
Combining this with \eqref{eq:imageEC},
\[
 \dim_\kappa L(C_p)=\rank_\kappa\bigl(E[\mathcal F]\bmod\mathfrak m\bigr).\qedhere
\]
\end{proof}

\subsection{Automorphisms of finite graphs}\label{subsec:graphs}

Now let $\Delta$ be a finite simple graph with vertex set $V(\Delta)$, adjacency matrix $A$, and all-one matrix $J$,
and let $x\in\Aut(\Delta)$.
As in the introduction, write $a_0(x)$ for the number of vertices fixed by $x$ and
\[
 a_1(x)=\bigl|\{u\in V(\Delta):u\sim x(u)\}\bigr|.
\]
The permutation matrix $T$ of $x$, defined by $Te_u=e_{x(u)}$, commutes with $A$ and with $J$, and satisfies
\begin{equation}\label{eq:traces}
 \tr(T)=a_0(x),\qquad \tr(TA)=a_1(x),\qquad \tr(TJ)=|V(\Delta)|.
\end{equation}

For $x$ of prime order $p$ with fixed vertex set $\mathcal F$, we write $A_{\mathcal F}$, $I_{\mathcal F}$,
$J_{\mathcal F}$ for the principal submatrices of $A$, $I$, $J$ on the rows and columns indexed by $\mathcal F$;
thus $A_{\mathcal F}$ is the adjacency matrix of the fixed-point subgraph $\Fix(x)$.

Theorem \ref{thm:local}, applied to the vertex set, takes the following form;
recall that the prime $p$ is the residue characteristic of $O$.

\begin{corollary}[Finite graphs]\label{cor:graph}
Let $x\in\Aut(\Delta)$ have prime order $p$, with fixed vertex set $\mathcal F$,
let $E\in\operatorname{End}_O\bigl(O^{V(\Delta)}\bigr)$ be an idempotent commuting with $T$, and put $L=EO^{V(\Delta)}$.
Then
\[
 \tr\bigl(x\mid K\otimes_OL\bigr)=\tr(TE)=\rank_\kappa\bigl(E[\mathcal F]\bmod\mathfrak m\bigr),
\]
where $E[\mathcal F]$ is the principal submatrix of $E$ on the rows and columns indexed by $\mathcal F$.
In particular, this applies to every spectral idempotent of $A$ whose matrix entries lie in $O$.
\end{corollary}

\begin{proof}
An endomorphism of $O^{V(\Delta)}$ commutes with the action of $C_p=\langle x\rangle$ if and only if it commutes with $T$,
so the equality of the two outer terms is Theorem \ref{thm:local} applied to the finite set $X=V(\Delta)$.
For the first equality, the summands $L$ and $(1-E)O^{V(\Delta)}$ are $T$-invariant,
and $TE$ vanishes on the second and agrees with $T$ on the first, so $\tr(TE)=\tr\bigl(x\mid K\otimes_OL\bigr)$.
Finally, $T$ commutes with every spectral idempotent of $A$: since $TAT^{-1}=A$,
conjugation by $T$ preserves each eigenspace of $A$, and therefore fixes each projection onto one eigenspace along the others.
\end{proof}

\begin{remark}\label{rem:walks}
If $E=f(A)$ with $f(t)=\sum_{i=0}^{n}c_it^i\in O[t]$, then $\tr(TE)=\sum_{i=0}^{n}c_iw_i(x)$, where
\[
 w_i(x)=\tr(TA^i)=\sum_{u\in V(\Delta)}(A^i)_{u,\,x(u)}
\]
counts the walks of length $i$ from a vertex to its image under $x$ --- so $w_0(x)=a_0(x)$ and $w_1(x)=a_1(x)$ by \eqref{eq:traces} --- and Corollary \ref{cor:graph} becomes
\[
 \sum_{i=0}^{n}c_iw_i(x)
 =\rank_\kappa\Bigl(\sum_{i=0}^{n}c_i\,(A^i)[\mathcal F]\bmod\mathfrak m\Bigr).
\]
The principal submatrix $(A^i)[\mathcal F]$ records walks in the ambient graph $\Delta$ whose endpoints are fixed by $x$;
it is not, in general, the $i$-th power of the adjacency matrix of the fixed-point subgraph.
For the strongly regular graphs considered next, the relevant idempotents are linear combinations of $I$, $A$, and $J$,
so the rank side is computed from $I_{\mathcal F}$, $A_{\mathcal F}$, $J_{\mathcal F}$ alone.
\end{remark}

\subsection{Strongly regular graphs}\label{subsec:srg}

Let $\Delta$ be a connected non-complete strongly regular graph with parameters $(v,k,\lambda,\mu)$,
as defined in Section \ref{sec:known}, with adjacency matrix $A$ and all-one matrix $J$, so that
\[
 A^2=(k-\mu)I+(\lambda-\mu)A+\mu J.
\]
Since $\Delta$ is connected and regular,
$k$ is a simple eigenvalue of $A$ whose eigenspace is spanned by the all-one vector,
and since $\Delta$ is not complete, the non-principal eigenvalues
\[
 r,s=\frac{(\lambda-\mu)\pm\sqrt{(\lambda-\mu)^2+4(k-\mu)}}2,\qquad r>s,
\]
are distinct.
For $\theta\in\{r,s\}$, write $\theta'$ for the other non-principal eigenvalue, so that $\theta'=\lambda-\mu-\theta$.
Over any field containing $\theta$, the spectral idempotent for $\theta$ is
\begin{equation}\label{eq:Ers}
 E_\theta=\frac{A-\theta'I-\frac{k-\theta'}{v}J}{\theta-\theta'}
 =\alpha_\theta I+\beta_\theta A+\gamma_\theta J,
 \qquad\alpha_\theta,\beta_\theta,\gamma_\theta\in\mathbb Q(\theta).
\end{equation}
Indeed, using $J\one=v\one$ and $Jw=0$ for $w$ orthogonal to $\one$,
the middle expression acts as the identity on the $\theta$-eigenspace and as zero on the other two eigenspaces;
moreover $E_r+E_s=I-J/v$, the remaining summand $J/v$ being the projection onto the $k$-eigenspace.

\begin{proposition}[Strongly regular graphs]\label{prop:srg-rank}
Let $\Delta$ be a connected non-complete strongly regular graph with parameters $(v,k,\lambda,\mu)$,
and let $x\in\Aut(\Delta)$ have prime order $p$, with fixed vertex set $\mathcal F$.
Fix $\theta\in\{r,s\}$, let $K$ be a finite extension of $\Qp$ containing $\alpha_\theta$, $\beta_\theta$,
$\gamma_\theta$, and let $O$ be its ring of integers, with maximal ideal $\mathfrak m$ and residue field $\kappa$.
Assume that $E_\theta$ preserves the lattice $O^{V(\Delta)}$ --- equivalently, that the entries of $E_\theta$ lie in $O$.
Then
\[
 \alpha_\theta\, a_0(x)+\beta_\theta\, a_1(x)+\gamma_\theta\, v
 =\rank_{\kappa}\bigl((\alpha_\theta I_{\mathcal F}+\beta_\theta A_{\mathcal F}+\gamma_\theta J_{\mathcal F})\bmod \mathfrak m\bigr).
\]
\end{proposition}

\begin{proof}
The ring $O$ is a complete discrete valuation ring with field of fractions $K$ of characteristic $0$ and residue field $\kappa$ of characteristic $p$,
so the conventions of this section apply to it.
The idempotent $E_\theta$ commutes with $T$, being a linear combination of $I$, $A$, and $J$,
each of which commutes with $T$, and by hypothesis it lies in $\operatorname{End}_O\bigl(O^{V(\Delta)}\bigr)$.
Corollary \ref{cor:graph} applied to $E=E_\theta$ gives
\[
 \tr(TE_\theta)=\rank_\kappa\bigl(E_\theta[\mathcal F]\bmod\mathfrak m\bigr);
\]
by \eqref{eq:traces} the left-hand side equals $\alpha_\theta\,a_0(x)+\beta_\theta\,a_1(x)+\gamma_\theta\,v$,
and $E_\theta[\mathcal F]=\alpha_\theta I_{\mathcal F}+\beta_\theta A_{\mathcal F}+\gamma_\theta J_{\mathcal F}$.
\end{proof}

\begin{remark}\label{rem:integral}
When $r$ and $s$ are integers one may take $K=\Qp$ and $O=\Zp$, so that $\kappa=\Fp$.
The hypothesis that $E_\theta$ preserve $\Zp^{V(\Delta)}$ is then the $p$-integrality of the rational numbers $\alpha_\theta$,
$\beta_\theta$, $\gamma_\theta$ --- equivalently, of the three entry values $\alpha_\theta+\gamma_\theta$ (diagonal),
$\beta_\theta+\gamma_\theta$ (adjacent pairs),
and $\gamma_\theta$ (non-adjacent pairs) --- and it is decided by their denominators: the proposition may be used at every prime $p$ dividing none of them,
and the reduction on the right-hand side may be carried out after multiplying by any integer prime to $p$ that clears the denominators,
without changing the rank.
This is the form used in all applications below.
When $r$ and $s$ are irrational one takes for $K$ the field $\Qp(\theta)$,
which is $\Qp$ itself or a quadratic extension of it according to whether $(\lambda-\mu)^2+4(k-\mu)$ is a square in $\Qp$,
and the rank is taken over its residue field.
\end{remark}

\begin{remark}\label{rem:benson}
The left-hand side of Proposition \ref{prop:srg-rank} is an ordinary character value,
and its a priori integrality is the source of the classical congruences for $a_1(x)$:
in the integral-eigenvalue case these are precisely the Benson-type congruences of De Winter--Kamischke--Wang \cite[Theorem 1]{DeWinterKamischkeWang},
and for the missing Moore graph they are the congruences of Ma\v{c}aj--\v{S}ir\'a\v{n} \cite[Lemma 3]{MacajSiran}.
Proposition \ref{prop:srg-rank} refines the congruence to an exact rank equality whenever $E_\theta$ is integral at $p$.
\end{remark}

\section{The spectral summands of the missing Moore graph}\label{sec:spectral}

We return to the missing Moore graph.
Throughout this section, $\Gamma$ is a Moore graph of diameter two and degree $57$,
$x\in\Aut(\Gamma)$ is an automorphism of prime order $p$ with permutation matrix $T$ and fixed vertex set $\mathcal F$, and
\[
 M=\Zp^{V(\Gamma)}
\]
is the vertex permutation lattice, as in Section \ref{sec:general} with $O=\Zp$.

The graph $\Gamma$ is a connected non-complete strongly regular graph whose non-principal eigenvalues $r=7$ and $s=-8$ are integral,
by \eqref{eq:spectrum}, so its spectral idempotents are given by \eqref{eq:Ers} with rational coefficients.
Substituting $(v,k)=(3250,57)$ and $(\theta,\theta')=(-8,7)$, respectively $(7,-8)$,
and clearing denominators gives the explicit forms
\begin{equation}\label{eq:E8-formula}
 E_{-8}=\frac{455I-65A+J}{975},\qquad 975=3\cdot 5^2\cdot 13,
\end{equation}
and
\begin{equation}\label{eq:E7-formula}
 E_{7}=\frac{400I+50A-J}{750},\qquad 750=2\cdot 3\cdot 5^3;
\end{equation}
equivalently, the coefficients of these idempotents with respect to the basis $I$, $A$,
$J-I-A$ of the adjacency algebra are recorded as rows of the matrix $Q$ in \cite[Theorem 1]{MacajSiran}.

If $p$ does not divide $975$,
formula \eqref{eq:E8-formula} exhibits $E_{-8}$ as an endomorphism of the $\Zp$-module $M$;
likewise $E_7$ is an endomorphism of $M$ whenever $p$ does not divide $750$.
These conditions are sharp: at a pair of distinct non-adjacent vertices,
the entry of $455I-65A+J$ is $1$ and the entry of $400I+50A-J$ is $-1$, so neither numerator matrix is divisible by any prime.
Summarizing, $E_{-8}$ is defined over $\Zp$ if and only if
\[
 p\notin\{3,5,13\},
\]
and $E_7$ is defined over $\Zp$ if and only if
\[
 p\notin\{2,3,5\}.
\]
In particular, for $p=2$ only $E_{-8}$ is available and for $p=13$ only $E_7$ is available, while for $p\in\{7,11,19\}$ both are.
For $p\in\{3,5\}$ neither idempotent is defined over $\Zp$;
this is why the primes $3$ and $5$ remain untouched by the method of this paper.

Fix a prime $p$ and let $E$ be one of $E_{-8}$, $E_7$, assumed defined over $\Zp$,
and write $\theta$ for the corresponding eigenvalue.
Since $E$ is an idempotent commuting with $T$ (being a linear combination of $I$, $A$, and $J$), the image
\begin{equation}\label{eq:Ldef}
 L_\theta=EM
\end{equation}
is a summand of the kind studied in Section \ref{sec:general}:
a $p$-permutation $\Zp C_p$-lattice whose scalar extension $\Qp\otimes_{\Zp}L_\theta$ is precisely the $\theta$-eigenspace of $A$,
by the spectral decomposition in \eqref{eq:spectrum}.

By \eqref{eq:traces}, $\tr(T)=a_0(x)$, $\tr(TA)=a_1(x)$, and $\tr(TJ)=3250$.
From the explicit forms \eqref{eq:E8-formula} and \eqref{eq:E7-formula},
\begin{equation}\label{eq:trace-E8}
 \tr(TE_{-8})=\frac{7a_0(x)-a_1(x)+50}{15},
\end{equation}
\begin{equation}\label{eq:trace-E7}
 \tr(TE_{7})=\frac{8a_0(x)+a_1(x)-65}{15}.
\end{equation}
Both right-hand sides are rational character values and make sense for every prime $p$;
the quantity \eqref{eq:trace-E7} is denoted $\chi_1(x)$ by Ma\v{c}aj--\v{S}ir\'a\v{n},
and its integrality already yields the congruence $a_1(x)\equiv 7a_0(x)+5\pmod{15}$ \cite[Lemma 3]{MacajSiran}.
Compare Remark \ref{rem:benson}.
Proposition \ref{prop:srg-rank} specializes as follows.

\begin{proposition}\label{prop:trace-rank}
Let $x\in\Aut(\Gamma)$ have prime order $p$, let $\mathcal F$ be its set of fixed vertices,
and let $A_{\mathcal F}$ be the adjacency matrix of the fixed-point subgraph.
If $p\notin\{3,5,13\}$, then
\[
 \frac{7a_0(x)-a_1(x)+50}{15}
 =\rank_{\Fp}\bigl(455I_{\mathcal F}-65A_{\mathcal F}+J_{\mathcal F}\bigr).
\]
If $p\notin\{2,3,5\}$, then
\[
 \frac{8a_0(x)+a_1(x)-65}{15}
 =\rank_{\Fp}\bigl(400I_{\mathcal F}+50A_{\mathcal F}-J_{\mathcal F}\bigr).
\]
\end{proposition}

\begin{proof}
For $E_{-8}$ the coefficients in Proposition \ref{prop:srg-rank} are
\[
 (\alpha,\beta,\gamma)=\Bigl(\tfrac{455}{975},\ -\tfrac{65}{975},\ \tfrac{1}{975}\Bigr),
\]
whose denominators involve exactly the primes $3$, $5$, $13$.
For $p\notin\{3,5,13\}$ the proposition applies with $K=\Qp$, as in Remark \ref{rem:integral};
its left-hand side equals \eqref{eq:trace-E8}, and, since $975$ is a unit in $\Zp$,
the rank on its right-hand side may be computed from the numerator matrix $455I_{\mathcal F}-65A_{\mathcal F}+J_{\mathcal F}$.
The second statement follows in the same way from $E_7$,
whose coefficients have denominator $750$ with prime divisors $2$, $3$, $5$.
\end{proof}

Note that the two left-hand sides in Proposition \ref{prop:trace-rank} sum to
\[
 \frac{\bigl(7a_0(x)-a_1(x)+50\bigr)+\bigl(8a_0(x)+a_1(x)-65\bigr)}{15}=a_0(x)-1,
\]
reflecting the spectral decomposition $I=E_{57}+E_{7}+E_{-8}$ with $E_{57}=J/3250$ and $\tr(TE_{57})=1$.
The two right-hand sides satisfy the corresponding relation: when both idempotents are defined over $\Zp$ --- that is,
for $p\notin\{2,3,5,13\}$
--- the reductions of $E_{-8}[\mathcal F]$ and $E_{7}[\mathcal F]$ modulo $p$ are the matrices of the orthogonal idempotents induced by $E_{-8}$ and $E_{7}$ on $\Fp^{\mathcal F}$,
as in the proof of Theorem \ref{thm:local}, so their ranks add up to the rank of the reduction of
\[
 E_{-8}[\mathcal F]+E_{7}[\mathcal F]=I_{\mathcal F}-J_{\mathcal F}/3250;
\]
and that rank equals $a_0(x)-1$, because the $\langle x\rangle$-orbits on the vertex set have size $1$ or $p$,
so that $a_0(x)\equiv 3250\not\equiv 0\pmod p$ and the reduction of $J_{\mathcal F}/3250$ is an idempotent of rank one.
Consequently, when both idempotents are defined over $\Zp$,
the two equalities of Proposition \ref{prop:trace-rank} determine one another,
and either one suffices to pin down $a_1(x)$ once the fixed-point subgraph is known.
For $p=2$ only the first equality is available, and for $p=13$ only the second.

\section{No involutions}\label{sec:involutions}

In this section we prove Theorem \ref{thm:main}.
Let $t\in\Aut(\Gamma)$ be an involution.
We apply Proposition \ref{prop:trace-rank} with $x=t$ and $p=2$; as noted in Section \ref{sec:spectral},
only the summand $L_{-8}$ is available at the prime $2$, because $750$ is even.
Integrality of the ordinary character values alone does not eliminate the involution:
Cameron remarks that an involution with $56$ fixed points ``does not seem so easy to eliminate by this method'' \cite[remark following Theorem 3.13]{Cameron}.
The rank equality of Proposition \ref{prop:trace-rank} does.

By Theorem \ref{thm:higman}, $\Fix(t)\cong K_{1,55}$, so $a_0(t)=56$.
Since $455$ and $65$ are odd, the matrix $455I_{\mathcal F}-65A_{\mathcal F}+J_{\mathcal F}$ reduces modulo $2$ to
\[
 I_{\mathcal F}+A_{\mathcal F}+J_{\mathcal F}.
\]
Order the vertices of $\mathcal F$ so that the centre of the star comes first.  Then over $\Ftwo$,
\[
 A_{\mathcal F}=
 \begin{pmatrix}
 0&\one^{\mathsf T}\\
 \one&0
 \end{pmatrix},
\]
where $\one$ has length $55$.  Hence
\begin{equation}\label{eq:block}
 I_{\mathcal F}+A_{\mathcal F}+J_{\mathcal F}=
 \begin{pmatrix}
 0&0\\
 0&I_{55}+J_{55}
 \end{pmatrix}
\end{equation}
over $\Ftwo$.  Since $55$ is odd,
\[
 \ker(I_{55}+J_{55})=\langle \one\rangle.
\]
Indeed, if $(I_{55}+J_{55})v=0$, then
\[
 v=(\one^{\mathsf T}v)\one,
\]
and $\one$ itself lies in the kernel because $55\equiv 1\pmod 2$.  Therefore
\begin{equation}\label{eq:rank54}
 \rank_{\Ftwo}\bigl(455I_{\mathcal F}-65A_{\mathcal F}+J_{\mathcal F}\bigr)
 =\rank_{\Ftwo}(I_{55}+J_{55})=54.
\end{equation}

By Proposition \ref{prop:trace-rank} and \eqref{eq:rank54},
\[
 \frac{7\cdot 56-a_1(t)+50}{15}=54,
 \qquad\text{that is,}\qquad
 a_1(t)=-368.
\]
This contradicts the non-negativity of the count $a_1(t)$.
Therefore no involution $t$ exists in $\Aut(\Gamma)$, and Theorem \ref{thm:main} is proved.

\begin{corollary}\label{cor:odd-order}
If a Moore graph of diameter two and degree $57$ exists, then its automorphism group has odd order.
\end{corollary}

\begin{proof}
By Cauchy's theorem, a finite group of even order contains an element of order $2$.
This is ruled out by Theorem \ref{thm:main}.
\end{proof}

\section{Restrictions for automorphisms of odd prime order}\label{sec:odd}

In this section we prove Theorem \ref{thm:odd}.
Throughout,
$x\in\Aut(\Gamma)$ is an automorphism of prime order $p\in\{7,11,13,19\}$ and $\mathcal F$ is its set of fixed vertices.
Lemma \ref{lem:fix-odd} determines the fixed-point subgraph in each case,
because exactly one of its six alternatives allows each of these primes: for $p=7$, case (3),
so $\Fix(x)$ is a star $K_{1,m}$ with $m=a_0(x)-1\equiv 1\pmod 7$; for $p=11$, case (4),
so $\Fix(x)$ is a pentagon and $a_0(x)=5$; for $p=13$, case (1), so $\Fix(x)$ is empty and $a_0(x)=0$; and for $p=19$,
case (2), so $\Fix(x)$ is a single vertex and $a_0(x)=1$.
For $p=7$, $11$, and $19$, both equalities of Proposition \ref{prop:trace-rank} are available, and we use the first.
This loses no information: by the observation at the end of Section \ref{sec:spectral},
the two equalities determine one another when both idempotents are defined over $\Zp$, their ranks summing to $a_0(x)-1$.
For $p=13$, only the second equality is available.
The four subsections below establish the four rows of the table in Theorem \ref{thm:odd}.

\subsection{Automorphisms of order 7}\label{subsec:order7}

Let $p=7$.
Both spectral idempotents are defined over $\mathbb Z_7$, since $975\equiv 2$ and $750\equiv 1\pmod 7$.
By Lemma \ref{lem:fix-odd}, $\Fix(x)$ is a star $K_{1,m}$ with $m=a_0(x)-1$ leaves and $a_0(x)\equiv 2\pmod 7$.

We compute the rank appearing in the first equality of Proposition \ref{prop:trace-rank}.
Modulo $7$ we have $455\equiv 0$ and $-65\equiv 5$, so
\[
 455I_{\mathcal F}-65A_{\mathcal F}+J_{\mathcal F}\equiv 5A_{\mathcal F}+J_{\mathcal F}\pmod 7.
\]
Order the vertices of $\mathcal F$ so that the centre of the star comes first.
Then over $\mathbb F_7$,
\[
 5A_{\mathcal F}+J_{\mathcal F}=uu^{\mathsf T},
 \qquad
 u=(1,-1,\dots,-1)^{\mathsf T},
\]
as one checks entry by entry: both sides have entry $1$ on the diagonal and between two leaves,
and entry $5+1=6\equiv -1$ between the centre and a leaf.
Hence the matrix has rank $1$, and Proposition \ref{prop:trace-rank} gives
\[
 \frac{7a_0(x)-a_1(x)+50}{15}=1,
 \qquad\text{that is,}\qquad
 a_1(x)=7a_0(x)+35.
\]

The exact value of $a_1(x)$ now feeds back into an elementary count which bounds $a_0(x)$ itself.
The powers $x$, $x^2$, and $x^3$ also have order $7$ and the same fixed star, so the equality just proved applies to each of them:
\[
 |S_k|=a_1(x^k)=7a_0(x)+35,
 \quad
 S_k=\{v\in V(\Gamma):v\sim x^k(v)\},
 \quad k=1,2,3.
\]
Applying the automorphism $x^{-k}$ shows that $v\sim x^k(v)$ if and only if $v\sim x^{7-k}(v)$,
so the three sets $S_1$, $S_2$, $S_3$ account for all six non-trivial powers of $x$.
A vertex $v$ lying in any $S_k$ is not fixed by any non-trivial power of $x$ (such a power generates $\langle x\rangle$,
and a fixed vertex is not adjacent to itself), so the seven vertices $x^i(v)$, $0\le i\le 6$, are pairwise distinct.
Triangle-freeness now makes the three sets pairwise disjoint:
\begin{itemize}
\item if $v\in S_1\cap S_2$, then $\{v,\,x(v),\,x^2(v)\}$ is a triangle,
the edge $x(v)\sim x^2(v)$ being the image of $v\sim x(v)$ under $x$;
\item if $v\in S_1\cap S_3$, then $v\in S_1\cap S_4$ and $\{v,\,x(v),\,x^4(v)\}$ is a triangle,
the edge $x(v)\sim x^4(v)$ being the image of $v\sim x^3(v)$ under $x$;
\item if $v\in S_2\cap S_3$, then $\{v,\,x^3(v),\,x^5(v)\}$ is a triangle: $v\sim x^3(v)$ since $v\in S_3$,
$v\sim x^5(v)$ since $v\in S_2=S_5$, and $x^3(v)\sim x^5(v)$ is the image of $v\sim x^2(v)$ under $x^3$.
\end{itemize}
Moreover, each $S_k$ is disjoint from the neighbourhood of the fixed star: if $u\in\mathcal F$, $v\sim u$,
and $v\sim x^k(v)$, then applying $x^k$ to $v\sim u$ gives $x^k(v)\sim u$, and $\{u,\,v,\,x^k(v)\}$ is a triangle.
Finally, let $v_1,\dots,v_m$ be the leaves of the star.
Two leaves are non-adjacent, and their unique common neighbour is the centre of the star, which is fixed;
hence the sets $N(v_i)\setminus\mathcal F$ are pairwise disjoint, and each has exactly $57-1=56$ elements,
because the only fixed neighbour of a leaf is the centre.

The sets $S_1$, $S_2$, $S_3$, $\mathcal F$,
and $N(v_i)\setminus\mathcal F$ for $i=1,\dots,m$ are therefore pairwise disjoint subsets of $V(\Gamma)$, whence
\[
 3\bigl(7a_0(x)+35\bigr)+a_0(x)+56\bigl(a_0(x)-1\bigr)\le 3250,
\]
that is, $78a_0(x)+49\le 3250$, so $a_0(x)\le 41$.  Combined with $a_0(x)\equiv 2\pmod 7$ this gives
\[
 a_0(x)\in\{2,9,16,23,30,37\},
\]
with the corresponding values $a_1(x)\in\{49,98,147,196,245,294\}$.
This proves the row $p=7$ of Theorem \ref{thm:odd}.

For comparison, at $p=7$ the table of \cite[Lemma 12]{MacajSiran} allows the values on the left;
Theorem \ref{thm:odd} retains only those on the right:
\[
\begin{array}{c|l|c}
 a_0(x) & a_1(x)\ \text{allowed by \cite[Lemma 12]{MacajSiran}} & a_1(x)\ \text{by Theorem \ref{thm:odd}} \\
\hline
 2 & 49+105k\le 500 & 49 \\
 9 & 98+105k\le 500 & 98 \\
 16 & 42+105k\le 500 & 147 \\
 23 & 91+105k\le 500 & 196 \\
 30 & 35+105k\le 500 & 245 \\
 37 & 84+105k\le 392 & 294 \\
 44 & 28+105k\le 260 & \text{excluded} \\
 51 & 77 & \text{excluded}
\end{array}
\]
In each surviving row the exact value selects one member of the arithmetic progression; the rows $a_0(x)\in\{44,51\}$,
which \cite[Lemma 12]{MacajSiran} leaves open (together with a row $a_0(x)=58$ that is already excluded there), disappear.
The resulting list of pairs coincides with the theorem of Kov\'a\v{c}ikov\'a \cite[Theorem 5.1]{Kovacikova}.
The proof there establishes, by computer-assisted counts of induced seven-vertex subgraphs of $\Gamma$,
that $a_1(x)$ must be divisible by $49$,
and eliminates the rows $a_0(x)\in\{44,51\}$ by comparison with the bounds in the table above.
(That $7a_0(x)+35=7\bigl(a_0(x)+5\bigr)$ is indeed divisible by $49$ follows from $a_0(x)\equiv 2\pmod 7$.)  The
derivation given here uses neither the computer search nor the tabulated bounds;
a form of the disjointness of the sets $S_k$ is already implicit in the proof of \cite[Lemma 12]{MacajSiran},
where it underlies the remark that an automorphism of order $7$ can contribute to at most one third of the orbits not connected with the fixed star.

\subsection{Automorphisms of order 11}\label{subsec:order11}

Let $p=11$, so that $\Fix(x)$ is a pentagon and $a_0(x)=5$.
Both spectral idempotents are defined over $\mathbb Z_{11}$, since $975\equiv 7$ and $750\equiv 2\pmod{11}$.
Order the five fixed vertices cyclically along the pentagon, so that its adjacency matrix $B$ is a circulant, as are $I$ and $J$.
Modulo $11$ we have $455\equiv 4$ and $-65\equiv 1$, so the matrix in the first equality of Proposition \ref{prop:trace-rank} is
\[
 455I_{\mathcal F}-65A_{\mathcal F}+J_{\mathcal F}\equiv 4I+B+J\pmod{11}.
\]
Since $5$ divides $11-1$, the field $\mathbb F_{11}$ contains primitive fifth roots of unity; take $\zeta=3$,
whose successive powers in $\mathbb F_{11}$ are $9$, $5$, $4$, $1$.
Circulant matrices of size $5$ are simultaneously diagonalized over $\mathbb F_{11}$ by the five vectors
\[
 w_j=(1,\zeta^j,\zeta^{2j},\zeta^{3j},\zeta^{4j})^{\mathsf T},\qquad j=0,\dots,4,
\]
on which $B$ acts by $\zeta^j+\zeta^{-j}$, and $J$ by $5$ for $j=0$ and by $0$ for $j\neq 0$.
The eigenvalues of $4I+B+J$ on $w_0,\dots,w_4$ are therefore
$4+2+5\equiv 0$ for $j=0$,
$\ 4+(3+4)\equiv 0$ for $j=\pm 1$,
and $4+(9+5)\equiv 7$ for $j=\pm 2$; exactly two of the five are nonzero, so
\[
 \rank_{\mathbb F_{11}}(4I+B+J)=2.
\]
Proposition \ref{prop:trace-rank} gives
\[
 \frac{7\cdot 5-a_1(x)+50}{15}=2,
 \qquad\text{that is,}\qquad
 a_1(x)=55.
\]

Among the values $a_1(x)\in\{55,220,385\}$ allowed by \cite[Lemma 12]{MacajSiran}, only the smallest survives.
This proves the row $p=11$ of Theorem \ref{thm:odd}.

\subsection{Automorphisms of order 13}\label{subsec:order13}

Let $p=13$, so that $\Fix(x)$ is empty and $a_0(x)=0$.
Here $13$ divides $975=3\cdot 5^2\cdot 13$,
so $E_{-8}$ is not defined over $\mathbb Z_{13}$ and the first equality of Proposition \ref{prop:trace-rank} is not available.
The idempotent $E_7$ is defined over $\mathbb Z_{13}$, since $750\equiv 9\pmod{13}$.

Since $\mathcal F$ is empty,
the matrix $400I_{\mathcal F}+50A_{\mathcal F}-J_{\mathcal F}$ is the empty matrix and its rank is $0$; equivalently,
$M(C_{13})=0$ by \eqref{eq:MC-basis}, so $L_7(C_{13})=0$.
The second equality of Proposition \ref{prop:trace-rank} gives
\[
 \frac{8\cdot 0+a_1(x)-65}{15}=0,
 \qquad\text{that is,}\qquad
 a_1(x)=65.
\]
Among the values $a_1(x)\in\{65,260,455\}$ allowed by \cite[Lemma 12]{MacajSiran}, only the smallest survives.
This proves the row $p=13$ of Theorem \ref{thm:odd}.

Substituting $a_0(x)=0$ and $a_1(x)=65$ into \eqref{eq:trace-E8} gives $\tr(TE_{-8})=-1$,
which is not the rank of any matrix: at $p=13$ the first equality of Proposition \ref{prop:trace-rank} genuinely fails,
not merely its proof, so the integrality hypotheses cannot be dropped.

\subsection{Automorphisms of order 19}\label{subsec:order19}

Let $p=19$, so that $\Fix(x)$ is a single vertex and $a_0(x)=1$.
Both spectral idempotents are defined over $\mathbb Z_{19}$, since $975\equiv 6$ and $750\equiv 9\pmod{19}$.
The matrix in the first equality of Proposition \ref{prop:trace-rank} is the $1\times 1$ matrix
\[
 455I_{\mathcal F}-65A_{\mathcal F}+J_{\mathcal F}=(456),
 \qquad
 456=24\cdot 19\equiv 0\pmod{19},
\]
of rank $0$.  The first equality of Proposition \ref{prop:trace-rank} gives
\[
 \frac{7\cdot 1-a_1(x)+50}{15}=0,
 \qquad\text{that is,}\qquad
 a_1(x)=57.
\]
Among the values $a_1(x)\in\{57,342\}$ allowed by \cite[Lemma 12]{MacajSiran}, only the smallest survives.
This proves the row $p=19$ of Theorem \ref{thm:odd} and completes the proof of the theorem.

\section{Further applications: sharpening automorphism analyses of hypothetical strongly regular graphs}\label{sec:applications}

Automorphisms of hypothetical strongly regular graphs are a well-studied subject.
Behbahani--Lam constrain automorphisms of prime order via orbit matrices \cite{BehbahaniLam};
Crnkovi\'c--Maksimovi\'c extend the orbit-matrix method to groups of composite order \cite{CrnkovicMaksimovic};
De Winter--Kamischke--Wang derive Benson-type congruences from ordinary eigenvalue traces \cite{DeWinterKamischkeWang};
Cesarz--Woldar give a computer-free analysis of the putative Conway $99$-graph \cite{CesarzWoldar};
and a series of papers by Makhnev and his coauthors determines,
for one parameter set at a time,
the possible fixed-point subgraphs of automorphisms of prime order by Higman's character-theoretic method \cite{MakhnevMinakova,MakhnevNosov2004,MakhnevNosov,MakhnevNirova,Nosov2016}.
In this section we use Proposition \ref{prop:srg-rank} to sharpen analyses of the latter kind.
Where a known fixed-point theorem determines the fixed subgraph but constrains the displacement count $a_1(x)$ only through the character congruences,
the rank equality replaces the congruence by an exact value;
for one parameter set the comparison moreover uncovers two arithmetic errors in the analysis of \cite{MakhnevNosov2004},
one of which re-opens a case believed closed (Remark \ref{rem:352-errata}).
The existence of each of the five parameter sets treated below is unresolved as of July 2026 \cite{BrouwerSRGTable,BrouwerVanMaldeghem}.

\subsection{Rank equations}\label{subsec:rank-eqs}

Let $\Delta$ be a strongly regular graph with parameters $(v,k,\lambda,\mu)$ and integral non-principal eigenvalues $r>s$,
and let $x\in\Aut(\Delta)$ have prime order $p$, with fixed vertex set $\mathcal F$.
Substituting the coefficients of \eqref{eq:Ers} into Proposition \ref{prop:srg-rank} and clearing denominators gives the two \emph{rank equations}
\begin{equation}\label{eq:rank-eqs}
 a_1(x)=(r-s)\,\rho_r+s\,a_0(x)+(k-s),
 \qquad
 a_1(x)=r\,a_0(x)+(k-r)-(r-s)\,\rho_s,
\end{equation}
where $\rho_\theta=\rank_{\Fp}\bigl(E_\theta[\mathcal F]\bmod p\bigr)$;
each equation is valid at every prime $p$ at which the entries of its idempotent are $p$-integral,
and the reduction may be carried out after multiplying by any integer prime to $p$ (Remark \ref{rem:integral}).
Every application below follows the same pattern:
a known fixed-point theorem determines the isomorphism type of $\Fix(x)$, or reduces it to a short list;
the rank $\rho_\theta$ of the resulting principal submatrix is computed;
and \eqref{eq:rank-eqs} returns $a_1(x)$ exactly.
Since the a priori integrality of the character value underlying \eqref{eq:rank-eqs} is precisely the classical congruence (Remark \ref{rem:benson}),
and the cited analyses constrain $a_1(x)$ through that congruence,
the passage from congruence to rank is exactly where the sharpening occurs.
Table \ref{tab:sharpening} summarizes the outcome.

\begin{table}[t]
\caption{Sharpenings obtained in this section.
Each line gives one branch:
the fixed subgraph, the constraint on $a_1(x)$ that the cited analysis leaves
(``known''; a dash if no valid constraint is available),
and the exact value produced by the rank equations \eqref{eq:rank-eqs} --- or the impossibility of the branch.
``Clebsch'' denotes the unique $\mathrm{srg}(16,5,0,2)$; for the name see Section \ref{subsec:mu2}.
Sources: $(162,21,0,3)$ \cite{MakhnevNosov}; $(352,26,0,2)$ \cite{MakhnevNosov2004};
$(1276,50,0,2)$ \cite{Nosov2016}; $(99,14,1,2)$ \cite{MakhnevMinakova}; $(640,243,66,108)$ \cite{MakhnevNirova}.}\label{tab:sharpening}
\begingroup
\footnotesize
\setlength{\arraycolsep}{2pt}
\[
\begin{array}{c|c|c|c|c}
 (v,k,\lambda,\mu) & p & \Fix(x) & \text{known }a_1(x) & a_1(x)\text{ here}\\
\hline
 (162,21,0,3) & 2 & \varnothing & 162 & \text{impossible}\\
 & & K_{1,3} & \text{---} & 12\\[2pt]
 (162,21,0,3) & 5 & K_2 & 15,\ 60 & 15\\[2pt]
 (352,26,0,2) & 11 & \varnothing & \text{--- (Remark \ref{rem:352-errata})} & 22\\[2pt]
 (352,26,0,2) & 13 & K_1 & \text{--- (Remark \ref{rem:352-errata})} & 26\\[2pt]
 (1276,50,0,2) & 5 & K_1 & \equiv 50 \pmod{70} & 50\\
 & & \text{Clebsch} & \equiv 0 \pmod{70} & 70\\[2pt]
 (1276,50,0,2) & 29 & \varnothing & 58,\ 464,\ 870,\ 1276 & 58\\[2pt]
 (99,14,1,2) & 3 & \varnothing & 18,\ 39,\ 60,\ 81 & 18\\
 & & K_3 & 6,\ 27,\ 48 & 6\\[2pt]
 (640,243,66,108) & 5 & \varnothing & 240,\ 480 & 240
\end{array}
\]
\endgroup
\end{table}

\subsection{\texorpdfstring{The parameter set $(162,21,0,3)$}{The parameter set (162,21,0,3)}}\label{subsec:162}

A strongly regular graph with $\lambda=0$ and $\mu=3$ is either $K_{3,3}$ or has parameters $(162,21,0,3)$,
and Makhnev--Nosov determined the possible fixed subgraphs of automorphisms of prime order of the latter \cite{MakhnevNosov}.
The eigenvalues are $r=3$ and $s=-6$, with
\[
 E_{3}=\tfrac23I+\tfrac19A-\tfrac1{54}J,
 \qquad
 E_{-6}=\tfrac13I-\tfrac19A+\tfrac1{81}J,
\]
so $E_3$ is available at every prime $p\notin\{2,3\}$ and $E_{-6}$ at every prime $p\neq 3$, and \eqref{eq:rank-eqs} reads
\[
 a_1(x)=9\rho_r-6a_0(x)+27,
 \qquad
 a_1(x)=3a_0(x)+18-9\rho_s.
\]

\begin{proposition}\label{prop:162-inv}
Suppose a strongly regular graph $\Delta$ with parameters $(162,21,0,3)$ exists, and let $t\in\Aut(\Delta)$ be an involution.
Then $\Fix(t)\cong K_{1,3}$ and $a_1(t)=12$.
\end{proposition}

\begin{proof}
By \cite[Theorem, statement (1)]{MakhnevNosov},
the fixed subgraph of an involution of $\Delta$ is either empty or the star $K_{1,3}$.
If $\Fix(t)=\varnothing$, the principal submatrix is empty, $\rho_s=0$,
and the second equation gives $a_1(t)=18$;
but since $\mu=3$ is odd, every vertex is adjacent to its image under a fixed-point-free involution \cite[proof of Lemma 2.2]{MakhnevNosov}, so $a_1(t)=162$ --- a contradiction.
Hence $\Fix(t)\cong K_{1,3}$ and $a_0(t)=4$.
Multiplying $E_{-6}$ by $81$ gives the integer matrix $27I-9A+J$, which reduces modulo $2$ to $I+A+J$.
On $K_{1,3}$, ordering the centre first,
the row and column of the centre vanish over $\Ftwo$ and the block on the three leaves is $J_3-I_3$,
whose rank over $\Ftwo$ is $2$.
Thus $\rho_s=2$ and the second equation gives $a_1(t)=12$.
\end{proof}

\begin{proposition}\label{prop:162-5}
Suppose a strongly regular graph $\Delta$ with parameters $(162,21,0,3)$ exists, and let $x\in\Aut(\Delta)$ have order $5$.
Then $\Fix(x)\cong K_2$ and $a_1(x)=15$.
\end{proposition}

\begin{proof}
$\Fix(x)\cong K_2$ by \cite[Theorem, statement (3)]{MakhnevNosov}, so $a_0(x)=2$.
Multiplying $E_3$ by $54$ gives $36I+6A-J$, whose entries on $K_2$ are $35$ on the diagonal and $5$ off it
--- both divisible by $5$, so $\rho_r=0$,
and the first equation gives $a_1(x)=15$.
\end{proof}

At order $7$ the identity recovers, in one line, the exact value $a_1(x)=21$ known from \cite[Lemma 2.3]{MakhnevNosov}:
the fixed subgraph is a single vertex, and the diagonal entry $35$ of $36I+6A-J$ vanishes modulo $7$.

\subsection{\texorpdfstring{The triangle-free parameter sets $(352,26,0,2)$ and $(1276,50,0,2)$}{The triangle-free parameter sets (352,26,0,2) and (1276,50,0,2)}}\label{subsec:mu2}

A strongly regular graph with $\lambda=0$ and $\mu=2$ has $k=u^2+1$ and $v=(u^4+3u^2+4)/2$ for a positive integer $u$;
the parameter sets $(352,26,0,2)$ and $(1276,50,0,2)$ ($u=5$ and $u=7$) are unresolved,
with fixed-point analyses available for both \cite{MakhnevNosov2004,Nosov2016}.
In both, the folded $5$-cube, the unique $\mathrm{srg}(16,5,0,2)$ \cite{BrouwerVanMaldeghem}, occurs as a possible fixed subgraph;
following \cite{MakhnevNosov2004,Nosov2016} we call it the \emph{Clebsch graph}
(in \cite{BrouwerHaemers,BrouwerVanMaldeghem} that name denotes its complement $\mathrm{srg}(16,10,6,6)$),
and the following computation disposes of all ranks it contributes.

\begin{lemma}\label{lem:clebsch-rank}
Let $N=\alpha I+\beta A+\gamma J$ with $\alpha,\beta,\gamma\in\mathbb Z$,
where $A$ is the adjacency matrix of the Clebsch graph $\mathrm{srg}(16,5,0,2)$, and let $p$ be an odd prime.
Then $N$ has eigenvalues $\alpha+5\beta+16\gamma$, $\alpha+\beta$, $\alpha-3\beta$ of multiplicities $1$, $10$, $5$,
and $\rank_{\Fp}(N\bmod p)$ is the sum of the multiplicities of those eigenvalues not divisible by $p$.
\end{lemma}

\begin{proof}
The Clebsch graph has spectrum $5^1$, $1^{10}$, $(-3)^5$,
and its three spectral idempotents
\[
 \tfrac1{16}J,
 \qquad
 E_1=\tfrac34I+\tfrac14A-\tfrac18J,
 \qquad
 E_{-3}=\tfrac14I-\tfrac14A+\tfrac1{16}J
\]
have entries with $2$-power denominators, hence are $p$-integral for every odd $p$.
The lattice $\Zp^{16}$ is therefore the direct sum of their images, of ranks $1$, $10$, $5$,
on which $N$ acts as the three stated scalars;
reducing modulo $p$, the rank of $N$ is the sum of the ranks of the summands whose scalar is a unit.
\end{proof}

For $(352,26,0,2)$ the eigenvalues are $r=4$ and $s=-6$, with
\[
 E_{4}=\tfrac35I+\tfrac1{10}A-\tfrac1{110}J,
 \qquad
 E_{-6}=\tfrac25I-\tfrac1{10}A+\tfrac1{160}J,
\]
available at $p\notin\{2,5,11\}$ and $p\notin\{2,5\}$ respectively, and \eqref{eq:rank-eqs} reads
\[
 a_1(x)=10\rho_r-6a_0(x)+32,
 \qquad
 a_1(x)=4a_0(x)+22-10\rho_s.
\]
Makhnev--Nosov analysed the automorphisms of this graph in \cite[\S2]{MakhnevNosov2004};
their Theorem 2 asserts that the possible prime orders are $2$, $5$, and $13$,
orders $7$ and $11$ being excluded by their Lemmas 2.5 and 2.2.

\begin{proposition}\label{prop:352}
Suppose a strongly regular graph $\Delta$ with parameters $(352,26,0,2)$ exists, and let $x\in\Aut(\Delta)$ have prime order $p\in\{11,13\}$.
\begin{enumerate}
\item[(1)] If $p=13$, then $\Fix(x)\cong K_1$ and $a_1(x)=26$.
\item[(2)] If $p=11$, then $\Fix(x)=\varnothing$ and $a_1(x^i)=22$ for $i=1,\dots,10$.
\end{enumerate}
\end{proposition}

\begin{proof}
By \cite[\S2]{MakhnevNosov2004} the fixed subgraph of an automorphism of odd prime order has order $0$, $1$, $2$, $4$, $16$, or $56$,
and $a_0(x)\equiv 352\pmod p$;
for $p=13$ this forces $a_0(x)=1$, and for $p=11$ it forces $a_0(x)=0$.

For $p=13$, multiplying $E_{-6}$ by $160$ gives $64I-16A+J$, whose single diagonal entry on $\Fix(x)$ is $65\equiv 0\pmod{13}$;
so $\rho_s=0$ and the second equation gives $a_1(x)=26$.
(The first equation, with the entry $65$ of $66I+11A-J$, gives the same value.)

For $p=11$ only $E_{-6}$ is available; the principal submatrix is empty, $\rho_s=0$,
and the second equation gives $a_1(x)=22$.
Each power $x^i$, $1\le i\le 10$, is again an automorphism of order $11$, so the same value applies to it.
\end{proof}

\begin{remark}\label{rem:352-errata}
Proposition \ref{prop:352} corrects the analysis of \cite{MakhnevNosov2004} in two places.
Both concern the integrality of the character value of \cite[Lemma 2.1]{MakhnevNosov2004},
which in our notation is $\tr(xE_{-6})=\bigl(4a_0(x)-a_1(x)+22\bigr)/10$.

(i) Lemma 2.3(1) of \cite{MakhnevNosov2004} asserts that $a_1(x^i)=52$ for an automorphism $x$ of order $13$.
This value is impossible: at $a_0(x)=1$ and $a_1(x)=52$ the character value is not an integer.
The proof there evaluates it as $(-13w+22)/10$ for $a_1(x)=13w$,
omitting the term $4a_0(x)=4$;
the correct integrality condition is $w\equiv 2\pmod{10}$,
satisfied by the actual value $w=2$ of Proposition \ref{prop:352}(1) and not by $w=4$.

(ii) Lemma 2.2 of \cite{MakhnevNosov2004} excludes automorphisms of order $11$ altogether.
Its proof derives, from the integrality at $a_0(x)=0$ and $a_1(x)=11w$, the congruence $w\equiv 8\pmod{10}$,
hence $a_1(x^i)\ge 88$ for each $i$ and a count of at least $40$ hendecagonal orbits among the $32$ orbits --- a contradiction.
But the correct integrality condition is $w\equiv 2\pmod{10}$,
with minimal value $w=2$;
by Proposition \ref{prop:352}(2) the actual value is $w=2$ for every power,
so exactly $10$ of the $32$ orbits are hendecagonal and the counting argument yields no contradiction.
Automorphisms of order $11$ therefore cannot be excluded by this route.
Consequently, the fixed-point-free case at $p=11$ must be restored in Theorem 2 of \cite{MakhnevNosov2004},
and the divisibility bound for the order of the automorphism group derived from it enlarged from $2^l\cdot5^2\cdot13$ to $2^l\cdot5^2\cdot11\cdot13$.
\end{remark}

The identity also confirms, independently, the exclusion of order $7$ in \cite[Lemma 2.5]{MakhnevNosov2004}:
the counting above allows $a_0(x)\in\{2,16\}$
--- a two-vertex clique, impossible since $x$ would act freely on the $25$ remaining neighbours of a fixed vertex, or the Clebsch graph \cite[\S2]{MakhnevNosov2004} ---
and in the Clebsch case Lemma \ref{lem:clebsch-rank} gives $\rho_s=10$,
so the second equation returns the negative count $a_1(x)=-14$.

For $(1276,50,0,2)$ the eigenvalues are $r=6$ and $s=-8$, with
\[
 E_{6}=\tfrac47I+\tfrac1{14}A-\tfrac1{308}J,
 \qquad
 E_{-8}=\tfrac37I-\tfrac1{14}A+\tfrac1{406}J,
\]
available at $p\notin\{2,7,11\}$ and $p\notin\{2,7,29\}$ respectively, and \eqref{eq:rank-eqs} reads
\[
 a_1(x)=14\rho_r-8a_0(x)+58,
 \qquad
 a_1(x)=6a_0(x)+44-14\rho_s.
\]
Nosov determined the fixed subgraphs of automorphisms of odd prime order \cite{Nosov2016};
at $p=7$ both idempotents are excluded and the identity is silent,
and at every other odd prime it gives the exact displacement count.

\begin{proposition}\label{prop:1276}
Suppose a strongly regular graph $\Delta$ with parameters $(1276,50,0,2)$ exists,
and let $x\in\Aut(\Delta)$ have odd prime order $p\neq 7$.
Then one of the following holds:
\begin{enumerate}
\item[(1)] $p=3$, $\Fix(x)$ is the Clebsch graph, and $a_1(x)=0$;
\item[(2)] $p=5$, and either $\Fix(x)\cong K_1$ and $a_1(x)=50$, or $\Fix(x)$ is the Clebsch graph and $a_1(x)=70$;
\item[(3)] $p=11$, $\Fix(x)=\varnothing$, and $a_1(x)=44$;
\item[(4)] $p=29$, $\Fix(x)=\varnothing$, and $a_1(x)=58$.
\end{enumerate}
\end{proposition}

\begin{proof}
The fixed subgraphs in (1)--(4) are those of \cite{Nosov2016}; it remains to compute the ranks.

For the empty fixed subgraph both ranks vanish,
and the equation available --- the second at $p=11$, the first at $p=29$ --- gives $a_1(x)=44$ and $58$ respectively.

For $p=5$ and $\Fix(x)\cong K_1$, the diagonal entry of $308E_6=176I+22A-J$ on the fixed vertex is $175\equiv 0\pmod 5$,
so $\rho_r=0$ and the first equation gives $a_1(x)=50$.

When $\Fix(x)$ is the Clebsch graph,
the matrix $308E_6=176I+22A_{\mathcal F}-J_{\mathcal F}$ has eigenvalues $270$, $198$, $110$ of multiplicities $1$, $10$, $5$ by Lemma \ref{lem:clebsch-rank}.
Modulo $5$ only $198$ is a unit, so $\rho_r=10$ and the first equation gives $a_1(x)=70$;
modulo $3$ only $110$ is a unit, so $\rho_r=5$ and the first equation gives $a_1(x)=0$.
\end{proof}

Parts (2) and (4) sharpen the congruences of \cite{Nosov2016} to exact values;
in the Clebsch case of part (2) the value $a_1(x)=0$ allowed by the congruence is excluded,
so some pentagonal orbit must occur.
We note a misprint in the one-vertex case of \cite[Lemma 2.2]{Nosov2016}:
its proof establishes $a_1(x)\equiv 50\pmod{70}$,
but the printed statement gives the pentagonal-orbit count as $2(7t-5)$ in place of $2(7t+5)$;
the exact value $50$ satisfies the congruence actually proved.
Parts (1) and (3) recover known values:
$a_1(x)=0$ at $p=3$ is elementary in a triangle-free graph, and $a_1(x)=44$ is the value of \cite[Lemma 2.4]{Nosov2016}.

\subsection{\texorpdfstring{The putative Conway $99$-graph}{The putative Conway 99-graph}}\label{subsec:conway}

A strongly regular graph with parameters $(99,14,1,2)$ is known as a Conway $99$-graph;
its existence is the subject of one of Conway's thousand-dollar prize problems (see \cite{CesarzWoldar}).
Makhnev--Minakova determined the fixed subgraphs of automorphisms of prime order \cite{MakhnevMinakova},
and Behbahani--Lam subsequently reduced the possible prime orders to $2$ and $3$ by an orbit-matrix computation \cite{BehbahaniLam}.
The eigenvalues are $r=3$ and $s=-4$, with
\[
 E_{3}=\tfrac47I+\tfrac17A-\tfrac2{77}J,
 \qquad
 E_{-4}=\tfrac37I-\tfrac17A+\tfrac1{63}J,
\]
available at $p\notin\{7,11\}$ and $p\notin\{3,7\}$ respectively, and \eqref{eq:rank-eqs} reads
\[
 a_1(x)=7\rho_r-4a_0(x)+18,
 \qquad
 a_1(x)=3a_0(x)+11-7\rho_s.
\]
At $p=7$ both idempotents are excluded and the identity is silent;
at the two prime orders that survive the computation of Behbahani--Lam it gives the following.

\begin{proposition}\label{prop:conway}
Suppose a Conway $99$-graph $\Delta$ exists, and let $x\in\Aut(\Delta)$ have order $3$.
Then either $\Fix(x)=\varnothing$ and $a_1(x)=18$, or $\Fix(x)\cong K_3$ and $a_1(x)=6$.
\end{proposition}

\begin{proof}
The two cases for $\Fix(x)$ are those of \cite{MakhnevMinakova}; only $E_3$ is available at $p=3$.
In the empty case $\rho_r=0$ and the first equation gives $a_1(x)=18$.
In the triangle case $a_0(x)=3$ and $A_{\mathcal F}=J_3-I_3$,
so the principal submatrix of $77E_3=44I+11A-2J$ is $33I_3+9J_3\equiv 0\pmod 3$;
hence $\rho_r=0$ and the first equation gives $a_1(x)=6$.
\end{proof}

At the remaining orders the identity recovers, in one line each,
the known exact values $a_1(t)=14$ for an involution and $a_1(x)=11$ at order $11$ \cite{MakhnevMinakova}.

\subsection{\texorpdfstring{The parameter set $(640,243,66,108)$}{The parameter set (640,243,66,108)}}\label{subsec:640}

Makhnev--Nirova determined the possible fixed subgraphs of automorphisms of prime order of a hypothetical strongly regular graph with parameters $(640,243,66,108)$ \cite{MakhnevNirova}.
In the fixed-point-free case their theorem leaves $p\in\{2,5\}$.
The eigenvalues are $r=3$ and $s=-45$, with
\[
 E_{3}=\tfrac{15}{16}I+\tfrac1{48}A-\tfrac3{320}J,
 \qquad
 E_{-45}=\tfrac1{16}I-\tfrac1{48}A+\tfrac1{128}J,
\]
excluded at $p\in\{2,3,5\}$ and $p\in\{2,3\}$ respectively:
$p=5$ is the only prime at which the identity applies to this parameter set at all,
and there it decides precisely the branch left open.
The second rank equation reads $a_1(x)=3a_0(x)+240-48\rho_s$.

\begin{proposition}\label{prop:640}
Suppose a strongly regular graph $\Delta$ with parameters $(640,243,66,108)$ exists,
and let $x\in\Aut(\Delta)$ have order $5$ with $\Fix(x)=\varnothing$.
Then $a_1(x)=240$.
\end{proposition}

\begin{proof}
The principal submatrix is empty, so $\rho_s=0$, and the equation above gives $a_1(x)=240$.
\end{proof}

\end{document}